\documentclass[11pt]{article}
\usepackage{amsfonts}
\usepackage{amsmath,amssymb}
\usepackage{algorithm}
\usepackage{algorithmicx}
\usepackage{algpseudocode}
\usepackage{yhmath}
\usepackage{qtree}
\usepackage{xcolor}
\usepackage{epsfig}
\usepackage{lineno,hyperref}
\usepackage{mathtools}
\usepackage{amsthm}
\usepackage[T1]{fontenc}
\usepackage[utf8]{inputenc}
\usepackage{booktabs}
\usepackage{authblk}
\usepackage[toc]{appendix}
\numberwithin{equation}{section}
\usepackage{verbatim}
\usepackage{graphicx}
\usepackage{geometry}
\usepackage[many]{tcolorbox}
\usepackage{caption}
\usepackage{color}
\usepackage{amsmath}
\allowdisplaybreaks[1]

\newcommand{\be}{\begin{equation}}
\newcommand{\ee}{\end{equation}}
\newcommand{\ba}{\begin{array}}
\newcommand{\ea}{\end{array}}
\newcommand{\bea}{\begin{eqnarray*}}
\newcommand{\eea}{\end{eqnarray*}}
\newcommand{\bean}{\begin{eqnarray}}
\newcommand{\eean}{\end{eqnarray}}

\newtheorem{theorem}{Theorem}[section]
\newtheorem{lemma}{Lemma}[section]
\newtheorem{remark}{Remark}[section]
\newtheorem{proposition}{Proposition}[section]

\newtheorem{corollary}{Corollary}[section]

\newcommand{\lc}{\mathrel{\raise2pt\hbox{${\mathop<\limits_{\raise1pt\hbox{\mbox{$\sim$}}}}$}}}
\newcommand{\gc}{\mathrel{\raise2pt\hbox{${\mathop>\limits_{\raise1pt\hbox{\mbox{$\sim$}}}}$}}}
\newcommand{\ec}{\mathrel{\raise1pt\hbox{${\mathop=\limits_{\raise2pt\hbox{\mbox{$\sim$}}}}$}}}

\newcommand{\nn}{\nonumber}

\begin{document}

\title{Higher-Order Energy-Decreasing Exponential Time Differencing Runge--Kutta methods for Gradient Flows}

\author[1]{Zhaohui Fu}
\author[2]{Jie Shen}
\author[3,4]{Jiang Yang}

\affil[1]{\small Department of Mathematics, National University of Singapore, Singapore ({fuzhmath@gmail.com})}
\affil[2]{\small School of Mathematical Science, Eastern Institute of Technology, Ningbo, Zhejiang, China {jshen@eitech.edu.cn})}
\affil[3]{\small Department of Mathematics, SUSTech International Center for Mathematics \& National Center for Applied Mathematics Shenzhen (NCAMS), Southern University of Science and Technology, Shenzhen, China  ({yangj7@sustech.edu.cn})}

\maketitle

\begin{abstract}
In this paper, we develop a general framework for constructing higher-order, unconditionally energy-stable exponential time differencing Runge–Kutta (ETDRK) methods applicable to a range of gradient flows. Specifically, we identify conditions sufficient for ETDRK schemes to maintain the original energy dissipation. Our analysis reveals that the widely employed third- and fourth-order ETDRK schemes fail to meet these conditions. To address this, we introduce new third-order ETDRK schemes, designed with appropriate stabilization, which satisfy these conditions and thus guarantee the unconditional energy decay property. We conduct extensive numerical experiments with these new schemes to verify their accuracy, stability, behavior under large time steps, long-term evolution, and adaptive time stepping strategy across various gradient flows. This study is the first to examine the unconditional energy stability of high-order ETDRK methods, and we are optimistic that our framework will enable the development of ETDRK schemes beyond the third order that are unconditionally energy stable.

\vskip .35cm
\textbf{Keywords:} exponential time differencing Runge--Kutta method; energy stability; gradient flows; phase-field models

\textbf{AMS subject classification (2020):} 65M12; 35K20; 35K35; 35K55

\end{abstract}


\section{Introduction}

We consider a  class of gradient flows  written in  the following general form:
\begin{equation}\label{gf} 
    u_t = G( \mathcal{L} u+f(u)), \quad (x,t)\in \Omega\times [0,T],
\end{equation}
 where $\Omega$ is a bounded domain in $R^d$ ($d=1,2,3$), $T$ is a finite time, $G$ is a nonpositive operator, and $\mathcal{L} $ is a positive definite operator. 
The above system is energy dissipative with the corresponding energy functional
\begin{equation}
    E(u)=\int_\Omega\left( \frac{1 }{2} |\mathcal{L}_{1/2} u |^2 +  F(u) \right)\,dx,
\end{equation}
where 
$F$ is such that $F'=f$. This general form covers a wide range of  gradient flows, such as the Allen--Cahn equation, the Cahn--Hilliard equation, the thin film model without slope selection also known as the molecular beam epitaxy (MBE) model, the phase-field crystal models, etc.  

Since these gradient flows involve strong nonlinearities and often  high-order derivatives, it is difficult to design an efficient time discretization scheme which is able to accurately approximate their dynamics and steady states. In particular, it is challenging for numerical schemes to guarantee the energy decay which is intrinsic to all of these models.  Ample numerical evidence indicates that non-physical oscillations may occur when the  energy stability is violated. 
Furthermore, the establishment of unconditional energy stability enables the creation of more efficient algorithmic designs, including the development of adaptive time stepping strategies, where the selection of time steps is solely governed by considerations of accuracy.


There have been a large number of studies concerning energy-stable schemes for gradient flows, see, e.g. \cite{gradsys,mixedvar,yang2015SDC,robust,ATA,IMEXAC,ACCH,IMEXBDF2varsteps,gSAVexpAC} and the references therein. In particular, the convex splitting technique \cite{convexsplitphasefield,convMBE} is useful to preserve unconditional energy stability, but it leads to a nonlinear scheme and is not easy to extend to more general cases. The invariant energy quadratization (IEQ) method (see, e.g., \cite{IEQ1,IEQ2}) and the  scalar auxiliary variable (SAV) method \cite{SAV1,SAV2} lead to linear schemes that can preserve different structures for a wide class of gradient flows. In particular, there have been a large number of works based on SAV methods, including SAV-BDF schemes \cite{SAVBDF,BDFlm} and SAV-RK schemes \cite{SAVRK,SAVRK2}. However, the energy stability achieved by the IEQ or SAV methods is based on a modified energy, not the original energy. 
We are not aware of any linear and higher than second-order for \eqref{gf} which is unconditional energy stable with respect to the original energy. 

On the other hand, exponential time differencing Runge--Kutta (ETDRK) schemes offer many advantages over multistep methods.
In this paper, we consider exponential time differencing Runge--Kutta (ETDRK) methods for \eqref{gf}. The key idea of ETDRK schemes is to apply the Duhamel's principle
\begin{equation}
    u(x,t)=e^{G\mathcal{L} (t-t_0)} u(x,t_0) - e^{G\mathcal{L} (t-t_0)} \int_{t_0}^{t} e^{-G\mathcal{L} (s-t_0)} G f(u(x,s)) ds,
\end{equation}
and approximate the implicit integral which contains the unknown $f(u(x,s))$ with an explicit RK method. The concept of exponential integrators has a long-standing history, tracing its origins back to the 1960s. There has been much literature related to using such methods to solve stiff problems, semilinear parabolic problems or gradient flows \cite{ETD1,ETDanal,ETDRK2,ETDRK2005a,ETDstiff}. However, ETDRK schemes are usually not energy stable. It has been shown in \cite{ETD1,MBPsemipara} that some first-order ETDRK schemes can be energy stable. More recently, it is shown in  \cite{ETDRK2} that a second-order ETDRK scheme, with proper stabilization,  is  unconditionally energy stable. However, no ETDRK scheme higher than second-order has proven to be energy stable.

This paper aims to introduce a comprehensive framework for developing unconditionally energy-stable ETDRK schemes tailored for various gradient flows. We have formulated a series of criteria for ensuring the unconditional energy stability of ETDRK schemes of arbitrary order. Our investigations reveal that the widely employed third-order and fourth-order ETDRK methods fall short of meeting these criteria. In response, we have designed innovative third-order ETDRK schemes that adhere to these standards, thereby achieving unconditional energy stability. To our knowledge, this work pioneers the establishment of sufficient conditions for high-order ETDRK methods to attain unconditional energy stability. We are optimistic that this foundational framework will pave the way for the creation of advanced ETDRK schemes surpassing the third order in terms of unconditional energy stability.

This paper is organized as follows. In Section \ref{sec2}, we present some background knowledge, and then state and prove our main theorem. In Section \ref{sec3}, we apply the main theorem to various known ETDRK schemes and determine whether they are unconditionally energy-stable, in particular, we show that the commonly used third-order and fourth-order ETDRK schemes do not satisfy these conditions, and we construct new third-order ETDRK schemes that satisfy these conditions and thus are unconditional energy stable. In Section \ref{sec4}, we  present ample numerical experiments to validate our new third-order schemes, and study how large time steps affect the solutions, how different time steps and stabilizers influence the energy evolution. We also present an example with adaptive time stepping to show the application of this scheme. Some concluding remarks are given in the final section.

\section{Energy decreasing ETDRK methods for gradient flows}\label{sec2}
In this section we first introduce some basic knowledge of gradient flows, phase field models, convex splitting and the ETDRK methods. Then we prove our main theorem, i.e. arbitrary ETDRK schemes are unconditionally energy stable as long as the conditions are satisfied. 

\subsection{Phase-field models}
We consider the general form of gradient flows in a Hilbert space $\mathcal{H}$:
\begin{equation}\label{gradientflows}
    u_t = G( \mathcal{L}  u+f(u)), \quad (x,t)\in \Omega\times [0,T],
\end{equation}
where $G$ is a nonpositive operator, $\mathcal{L} $ is a positive definite operator, and $f(u)=F'(u)$, where $F(u)$ is a nonlinear functional. In the following text, we denote the inner product by $(\cdot,\cdot)$. Taking the inner product of \eqref{gradientflows} with $ \mathcal{L}  u+f(u)$, we find the energy dissipation law:
\begin{equation}
	\frac{d}{dt} E(u)=(G( \mathcal{L}  u+f(u)),  \mathcal{L}  u+f(u))\le 0,
\end{equation}
where the free energy $E(u)$ is given by
\begin{equation}\label{EAC}
	E(u)=\int_\Omega \left(\frac{1}{2} (\mathcal{L}u,u) +F(u)\right) dx.
\end{equation}
 
Some typical examples are provided below:
\begin{itemize}
    \item the Allen--Cahn equation, $G=-I, \mathcal{L} =-\epsilon^2\Delta, f(u)=u^3-u$ and $ F(u)=\frac{1}{4} (u^2-1)^2$;
    \item the Cahn--Hilliard equation, $G=\Delta, \mathcal{L} =-\epsilon^2\Delta, f(u)=u^3-u$ and $ F(u)=\frac{1}{4} (u^2-1)^2$;
    \item the MBE model without slope selection, $G=-I, \mathcal{L} =-\epsilon^2\Delta^2$ but $f(u)$ should be replaced by $f(\nabla u)=-\nabla \cdot(\frac{\nabla u}{1+|\nabla u|^2})$, and correspondingly $ F(\nabla u)=-\frac{1}{2} \ln{( |\nabla u|^2+1)}$;
    \item the phase field crystal equation, $G=\Delta, \mathcal{L}=(\Delta+1)^2, f(u)=u^3-\epsilon u$ and $ F(u)=\frac{1}{4} (u^2-\epsilon)^2$.
\end{itemize}

We shall assume that $f$ is a Lipschitz continuous function with the Lipschitz constant $C_L$, i.e. we have \begin{equation}\label{lip}
	\|f(u)-f(v)\|\leq C_L \|u-v\|,\; \forall\; u,\;v \in \mathcal{H}. 
	\end{equation}
For the MBE model, the functions $F$ and $f$ depend only on $\nabla u$, which is different from other examples. The Lipschitz condition is satisfied in the sense
\begin{equation}\label{LipMBE}
	\| \partial^2_{\nabla u} F(\nabla u) \|_2 \leq 1,
\end{equation}
which also bounds the growth of the nonlinear term.

Notice that for the Cahn--Hilliard equation and the phase field crystal model, $f(u)$  does not naturally satisfy the Lipschitz assumption on nonlinearity, but we can truncate $f(u)$ to quadratic growth as in \cite{ACCH,ETDRK2} so as to satisfy the Lipschitz assumption. In practice such a modification never affects anything because of the boundedness of their solutions, see \cite{stabsizeforsemiFrCH,convganalfornonlocalCH}, which offer solid analysis without Lipschitz assumptions.

Consider the natural splitting of the energy $E(u)=E_l-E_n$ with
\begin{equation}
    \begin{aligned}
    & E_l(u)=\int_\Omega\left(  \frac{1 }{2}  |\mathcal{L} _{1/2} u |^2 + \frac{\beta}{2} |u|^2 \right)\ dx, \\
    & E_n(u)=\int_\Omega\left(  -F(u) + \frac{\beta}{2} |u|^2  \right)\ dx.
    \end{aligned}
\end{equation}
From this perspective, the gradient flow can thus be written as
\begin{equation}\label{sgf} 
    u_t=G(Lu-g(u)),
\end{equation}
where $L=\beta I+\mathcal{L} $  and $g=\beta I-f$ correspond to $E_l$ and $E_n$, respectively. In the computation, $Lu$ is treated implicitly and $g(u)$ is treated explicitly, which leads to linearly implicit schemes. 

Here $\beta$ serves as a stabilization to enhance the dissipation of linear part, so as to bound the Lipschitz growing nonlinear term in the analysis. In fact, such modification or stabilization is necessary for the proof of energy decay, which is pointed out in the remark attached to the main theorem. However, the stabilization may cause significant time delay phenomena if a low-order scheme is used or a large time step is taken. We will study the effect in detail for higher-order schemes in Section \ref{sec4}. On the other hand, the stabilization is necessary to guarantee the maximum bound property (MBP) for Allen--Cahn equations \cite{ETD1}. Besides, \cite{SAV2} provides numerical evidence to illustrate that the stabilization can significantly improve the numerical performance.

\subsection{ETDRK methods}
The key idea of ETDRK is to consider Duhamel's principle for the equation \eqref{sgf}
\begin{equation}
    u(x,t)=e^{GL(t-t_0)} u(x,t_0) - e^{GL(t-t_0)} \int_{t_0}^{t} e^{-GL(s-t_0)} G g(u(x,s)) ds,
\end{equation}
and approximate the implicit integral with a suitable quadrature formula. For example, assuming that we have $u_n$, the approximate solution at time step $n$, the simplest way to determine $u_{n+1}$  is to substitute the function $g(u)$ by a constant $g(u_n)$ which leads to the first-order ETD (ETD1) scheme:
\begin{equation}
    u_{n+1}=e^{\tau GL} u_n + \left(I-e^{\tau GL}\right)L^{-1}g(u_n),
\end{equation}
where $\tau$ is the time step.
For this scheme, the energy dissipation law and MBP for the Allen--Cahn equation have been proved, see, e.g., \cite{MBPsemipara}. 
The classical second-order ETDRK (ETDRK2) for the gradient flow (\ref{gf}) reads
\begin{eqnarray}
&& v=e^{\tau GL} u_n + \left(I-e^{\tau GL}\right)L^{-1}g(u_n),\\
&& u_{n+1}=v-\frac{1}{\tau} \left(e^{\tau GL}-I-\tau GL\right)(GL)^{-2}(Gg(v)-Gg(u_n)).
\end{eqnarray}
It is shown in \cite{ETDRK2} that the above scheme is energy decreasing with a suitably large stabilization constant $\beta$ for the Allen-Cahn and Cahn-Hilliard equations.

In general, the ETDRK schemes take the following form \cite{ETDstiff}:
\begin{equation}\label{ETDRKoriginform}
    \begin{aligned}
         &v_1=u_n,\\
         &v_i=\chi_i (\tau GL) u_n - \tau \sum_{j=1}^{i-1} a_{ij}(\tau GL) Gg(v_j),\quad i=2,...,s ,\\
         &u_{n+1}=\chi (\tau GL) u_n - \sum_{j=1}^s b_j(\tau GL) Gg(v_j),
    \end{aligned}
\end{equation}
where $\chi(z)=e^z,$ $\chi_i(z)=\chi(c_i z)$, and the coefficients $a_{ij},b_j$ are constructed to equal to or approximate exponential functions. For simplicity we define a class of functions which will be frequently used
\begin{equation}
  \phi_0 (z) = e^z,\quad \phi_{k+1}(z) = \frac{\phi_k(z)-\phi_k(0)}{z},\; \text{ with }\; \phi_k(0)=1/k!.
\end{equation}
The ETDRK schemes \eqref{ETDRKoriginform} could also be written in Butcher Tableau, although now its coefficients are functions
\begin{equation}\label{Tableau}
\begin{array}{c|cccc|c}
    c_1 &        &     &  &   &\chi_1(\tau GL)\\
    c_2 & a_{21} &     &  &   &\chi_2(\tau GL)\\
    ... & ...    & ... &  &   &...\\
    c_s & a_{s1} & ... & a_{s,s-1}&  &\chi_s(\tau GL)\\
    \hline
        & b_1 & b_2 & ... & b_s & \chi(\tau GL)
\end{array} 
\end{equation}
In order to preserve the equilibria, the coefficients of the method have to satisfy 
\begin{equation}\label{equilcond}
    \sum_{j=1}^s b_j(z)=\frac{\chi(z)-1}{z}, \quad \sum_{j=1}^s a_{ij}(z)=\frac{\chi_i(z)-1}{z}.
\end{equation}
Taking the ETDRK2 as an example, $a_{21}(z)=\phi_1(z)=(e^z -1)/z, b_2(z)=\phi_2(z)=(e^z-z-1)/z^2$ and $b_1=\phi_1-\phi_2$, where $z=\tau GL$ and its Butcher Tableau reads
\begin{equation}
	\begin{array}{c|cc|c}
		0   &        &        & 1\\
		1   & \phi_1(\tau GL) &        &\chi(\tau GL)\\
		\hline
		& \phi_1(\tau GL)-\phi_2(\tau GL) & \phi_2(\tau GL)  & \chi(\tau GL)
	\end{array} 
\end{equation}

With the help of (\ref{equilcond}) and setting $u_{n+1}=v_{s+1}$, we can rewrite the solution (\ref{ETDRKoriginform}) as
\begin{equation}\label{ETDRKrewritten}
    \begin{aligned}
        & v_i = u_n +\tau \sum_{j=1}^{i-1} a_{ij} (\tau GL) (G Lu_n-G g(v_j)),\quad i=2,...,s,\\
        & v_{s+1}= u_n +\tau \sum_{j=1}^s b_j(\tau GL) (G Lu_n-G g(v_j)).
    \end{aligned}
\end{equation}

\subsection{A general framework for energy stable ETDRK schemes}
We present below a general framework for constructing energy-stable ETDRK schemes.
We start with a useful lemma whose proof is straightforward.

\begin{lemma}
Consider a positive-definite operator $L=\beta I+\mathcal{L} $, and let $f$ be an analytic function whose domain includes the spectrum of $L$, i.e., the values
$\{f(\lambda_i)\}_{ i\in \mathcal{N}}$
exist, where $\{\lambda_i\}_{i\in\mathcal{N}}$ are the eigenvalues of $L$. Then, the eigenvalues of $f(L)$ are $\{f(\lambda_i)\}_{i\in\mathcal{N}}$. 
Furthermore, if $f$ is a positive function, then $f(L)$ is also a positive-definite operator.
\end{lemma}
   Hereafter,  we say a matrix $\Delta$ is positive-definite if the eigenvalues of its symmetrizer $(\Delta+\Delta^T)/2$ are all positive.
\begin{theorem}\label{mainth}
Consider the gradient flow 
\begin{equation}
    u_t = G( \mathcal{L}  u+f(u)), \quad (x,t)\in \Omega\times [0,T]
\end{equation}
where $G$ is a negative operator, $\mathcal{L} $ is a sectorial operator and $f$ is a Lipschitz continuous function with the Lipschitz constant $C_L$. The ETDRK schemes \eqref{ETDRKoriginform} with the stabilizer $\beta\geq C_L$ unconditionally decreases the energy as long as the following determinant 
\begin{equation}
    \Delta(z)=z E_L+P^{-1} E_L-\frac{z}{2} I,
\end{equation}
 is positive-definite for all negative $z\in \mathbb{R}$,
where $I$ is the identity operator, $E_L=(1_{i\geq j})_{s\times s}$ is the lower triangular matrix with all nonzero entries equal to 1, and 
\begin{equation}
    P=\begin{pmatrix}
        a_{21} & & & \\
        a_{31} & a_{32} & &  \\
        \vdots & \vdots & \ddots &  \\
        a_{s1} & a_{s2} & \cdots & a_{s(s-1)} \\
        b_1 & b_2 & \cdots & b_{s-1} &  b_s \\
    \end{pmatrix},
\end{equation}
where $a_{ij}$ and $b_i$ are given in \eqref{Tableau}.
\end{theorem}

\begin{proof}

We first compute the difference in the energy and derive a key inequality. 

Since the function $f$ is Lipschitz continuous, given any $v, u$,  we have
\begin{equation}\label{lipineq}
\begin{aligned}
        \left(F(v)-F(u),1\right) &\leq (v-u,f(u))+\frac{C_L}{2} (v-u,v-u)\\
        &=- (v-u,g(u))+\beta (v-u, u) +\frac{C_L}{2} (v-u,v-u).
\end{aligned}
\end{equation}
On the other hand,
\begin{equation}
    \begin{aligned}
            \frac{1}{2} \left(\int_\Omega |\mathcal{L} _{1/2} v|^2-|\mathcal{L} _{1/2} u|^2 dx\right) &= \frac{1}{2} \left((v,\mathcal{L}  v)-(u,\mathcal{L}  u)\right)  \\
            &=(v-u,\mathcal{L}  v)-\frac{1}{2} (v-u,\mathcal{L} (v-u))\\
            &=(v-u,Lv) - \beta (v-u,v) -\frac{1}{2} (v-u,\mathcal{L} (v-u)),
    \end{aligned}
\end{equation}
where we  used the identity 
$$[a,a]-[b,b]=2[a-b,a]-[a-b,a-b]$$
which is valid for all $a,b$ and any bilinear form $[\cdot,\cdot]$. Therefore, combining these two parts we derive
\begin{equation}\label{Lv-g}
\begin{aligned}
    E(v)-E(u)\leq (v-u,Lv- g(u))-\frac{1}{2} (v-u, L(v-u)) - \frac{\beta-C_L}{2} (v-u, v-u).
\end{aligned}
\end{equation}
This vital inequality holds with the Lipschitz condition and if we take $\beta\geq C_L$, the second term is naturally non-positive. 

Now we focus on the ETDRK scheme (\ref{ETDRKrewritten}), the first line of which is $v_1=u_n$, and the rest can be rewritten as the following system:
\begin{equation}
        \begin{pmatrix}
        v_2-v_1 \\ v_3-v_1 \\ ...\\v_{s+1}-v_1
    \end{pmatrix} = \tau P 
    \begin{pmatrix}
        G(Lv_1-g(v_1)) \\ G(Lv_1-g(v_2)) \\ ...\\G(Lv_1-g(v_s))
    \end{pmatrix},
\end{equation}
which is equivalent to
\begin{equation}\label{Lv-g2}
        \begin{pmatrix}
        Lv_1-g(v_1) \\ Lv_1-g(v_2) \\ ...\\Lv_1-g(v_s)
    \end{pmatrix}=
    \frac{1}{\tau}  P^{-1}
    \begin{pmatrix}
                G^{-1}(v_2-v_1) \\ G^{-1}(v_3-v_1) \\ ...\\G^{-1}(v_{s+1}-v_1)
    \end{pmatrix}=  \frac{1}{\tau} P^{-1} E_L
    \begin{pmatrix}
        G^{-1}(v_2-v_1) \\ G^{-1}(v_3-v_2) \\ ...\\G^{-1}(v_{s+1}-v_s)
    \end{pmatrix}.
\end{equation}
Therefore,
\begin{equation}\label{ineqtheorem}
    \begin{aligned}
        E(u_{n+1})&-E(u_n)
        =\sum_{k=1}^s E(v_{k+1})-E(v_{k}) \\
        \leq& \sum_{k=1}^s \left( v_{k+1}-v_{k},Lv_{k+1}-g(v_{k}) \right)
          - \frac{1}{2} \sum_{k=1}^s (v_{k+1}-v_{k},L(v_{k+1}-v_{k})) - \frac{\beta-C_L}{2} \sum_{k=1}^s \|v_{k+1}-v_{k}\|^2\\
        =&\sum_{k=1}^s \left( (v_{k+1}-v_k,Lv_{k+1}-Lv_1)+(v_{k+1}-v_k,Lv_1-g(v_k) \right) \\
        &- \frac{1}{2}\sum_{k=1}^s (v_{k+1}-v_{k},L(v_{k+1}-v_{k}))- \frac{\beta-C_L}{2} \sum_{k=1}^s \|v_{k+1}-v_{k}\|^2 \quad \quad (\text{using} \ (\ref{Lv-g2}))\\
        =&\frac{ 1}{\tau}\sum_{k=1}^s \sum_{j=1}^k (v_{k+1}-v_k,G^{-1}(\tau G L) (v_{j+1}-v_j))+(v_{k+1}-v_k,G^{-1} (P^{-1} E_L)_{kj} (v_{j+1}-v_j))\\
        & - \frac{1}{2\tau} \sum_{k=1}^s (v_{k+1}-v_{k},G^{-1} (\tau GL) (v_{k+1}-v_{k}))- \frac{\beta-C_L}{2} \sum_{k=1}^s \|v_{k+1}-v_{k}\|^2\\
        =&\frac{1}{\tau}\sum_{k,j=1}^s (v_{k+1}-v_k,G^{-1}\Delta_{kj}(\tau GL) (v_{j+1}-v_j))- \frac{\beta-C_L}{2} \sum_{k=1}^s \|v_{k+1}-v_{k}\|^2,
    \end{aligned}
\end{equation}
where $\Delta(z)=z E_L+P^{-1} E_L-\frac{z}{2} I$ with $ z=\tau G L$. Since $G$ is a negative operator and $\beta\geq C_L$, the discrete energy unconditionally decreases if $\Delta(z)$ is positive-definite for all negative $z\in R$.

\end{proof}

\begin{remark}
For the MBE model, recall that the nonlinear term is Lipschitz continuous as a function of $\nabla u$, and meanwhile the convex splitting of the energy is also different (see \cite{epitaxy4,epxywithornot} for more details). However, all analysis in the proof can be carried out in the similar way. To keep the presentation short, we omit the proof. 
\end{remark}

\begin{remark}\label{convsprmk}
    The inequality (\ref{Lv-g}) in the proof plays a very important role. It is the only place where we apply the Lipschitz condition. Alternatively, we can also replace the Lipschitz condition with a convex splitting approach since
    \begin{equation}
\begin{aligned}
        E(v)-E(u)
        &\leq \left(v-u, \frac{\delta E_l}{\delta v} (v)-\frac{\delta E_n}{\delta u} (u)\right)\\
        &=\left(v-u,\mathcal{L}v+f(u)\right)\\
        &=\left(v-u,Lv-g(u)\right).
\end{aligned}
\end{equation}
The determinant of this version is slightly different, while the rest of the proof is the same. 
\end{remark}
\begin{lemma}
    The positive-definiteness of $\frac 12 (\Delta+\Delta^T)$ is equivalent to the positive-definiteness of 
    \begin{equation*}
        \Delta' = z P E P^T + E_L P^T + P E_L^T.
    \end{equation*}
    \begin{proof}
    Notice that
        \begin{equation*}
            E_L + E_L^T = E + I,
        \end{equation*}
        and
        \begin{equation*}
            P (\Delta+\Delta^T) P^T  =  P (E_L+E_L^T-I)P^T +E_L P^T + P E_L^T = \Delta'.
        \end{equation*}
    \end{proof}
\end{lemma}
\begin{remark}
$\Delta'$ is already symmetric and helps us in the understanding and subsequent proof regarding positive-definiteness.
\end{remark}

\section{Examples of energy decreasing ETDRK  schemes}\label{sec3}

In this section we consider some examples of ETDRK methods up to order four.

 We first recall
\begin{equation}
  \phi_0 (z) = e^z,\quad \phi_{k+1} = \frac{\phi_k(z)-\phi_k(0)}{z},\quad \phi_k(0)=\frac{1}{k!}.
\end{equation}
We will also denote $\phi_{i,j}(z)=\phi_i(c_j z)$ for simplicity.

\subsection{First-order ETD scheme}
We start with the very simple ETD1 scheme
\begin{equation}
        u_{n+1}=e^{\tau GL} u_n + \left(I-e^{\tau GL}\right)L^{-1}g(u_n),
\end{equation}
whose Butcher Tableau reads
\begin{equation}
\begin{array}{c|c}
    0   &    0            \\
    \hline
        & \phi_1(\tau GL)
\end{array}.
\end{equation}

\begin{corollary}
    The first-order ETD scheme unconditionally decreases the discrete energy.
\end{corollary}
\begin{proof}
    The only element in $P$ is $b_1=\phi_1$, whence the determinant $\Delta=z+1/\phi_1(z)-z/2=z e^z / (e^z-1)-z/2$. Therefore, $\Delta\geq 0$ always holds for $z\leq 0$. The theorem (\ref{mainth})  guarantees that the discrete energy of the ETD1 solution unconditionally decreases.
\end{proof}

\subsection{Second-order ETDRK scheme}
In \cite{ETDRK2} it has been proven that the following ETDRK2 scheme is unconditionally energy stable, while here we could make use of the framework to obtain the same result. In fact, the estimate here is finer than the result in \cite{ETDRK2}. 

We consider the following ETDRK2 scheme
\begin{eqnarray}
&& v=e^{\tau GL} u_n + \left(I-e^{\tau GL}\right)L^{-1}g(u_n),\\
&& u_{n+1}=v-\frac{1}{\tau} \left(e^{\tau GL}-I-\tau GL\right)(GL)^{-2}(Gg(v)-Gg(u_n)).
\end{eqnarray}
The Butcher Tableau reads
\begin{equation}
\begin{array}{c|cc}
    0   &    0     &       \\
    1   &    \phi_1     &       \\
    \hline
        & \phi_1-\phi_2  &\phi_2
\end{array}.
\end{equation}
\begin{corollary}
    The second-order ETDRK scheme unconditionally decreases the discrete energy.
\end{corollary}
\begin{proof}
    According to the Butcher Tableau, we have
\begin{equation}\nonumber
    P=\begin{pmatrix}
        \phi_1 && 0 \\
        \phi_1-\phi_2 && \phi_2
    \end{pmatrix}, \quad
    P^{-1}=\begin{pmatrix}
        1/\phi_1 && 0 \\
         (\phi_2-\phi_1)/(\phi_1 \phi_2)    &&  1/\phi_2
    \end{pmatrix},
\end{equation}
whence the determinant reads
\begin{equation}
    \Delta=\begin{pmatrix}
        z+1/\phi_1 && 0 \\
        z+1/\phi_1 && z+1/\phi_2
    \end{pmatrix} -\frac{z}{2} I_{2} =
        (z+1/\phi_1) E_L+ \begin{pmatrix}
            0 && 0 \\
            0 && 1/\phi_2-1/\phi_1
        \end{pmatrix} -\frac{z}{2} I_{2}.
\end{equation}
The first term is positive-definite since $(z+1/\phi_1)$ is positive and $E_L$ is positive-definite, the second term is also positive-definite because $\phi_1>\phi_2>0$, and the third term is obviously positive-definite. Therefore, $\Delta$ is positive-definite, and the ETDRK2 scheme is unconditionally energy stable.
\end{proof}

\subsection{Third-order ETDRK schemes}

In general, third-order ETDRK schemes need to satisfy the following order conditions:
\begin{center}\label{ordcond3}
\begin{tabular}{|c|c|}
\hline
  Order & Conditions \\ 
  \hline
  1 & $\psi_1=0$ \\
  \hline
  2 & $\psi_2=0$  \\
  2 & $\psi_{1,i}=0$ \\
  \hline
  3 & $\psi_3=0$ \\
  3 & $\sum_{i=1}^s b_i J \psi_{2,i}=0$ \\
\hline
\end{tabular}
\end{center}
where 
\begin{equation}\label{repsi}
\psi_i(z)=\phi_i(z)-\sum_{k=1}^s b_k c_k^{j-1}/(j-1)!,\quad \psi_{i,j}=\phi_i c_j^i-\sum_{k=1}^{j-1} a_{jk} c_k^{i-1}/(i-1)!, 
\end{equation}
and $J$ denotes arbitrary bounded operators. In particular,  the classical ETDRK3 scheme below from Cox and Matthews \cite{ETDstiff}  does not satisfy the conditions of our theorem (see Fig. \ref{eigenvalueETDRK3}) so it is not unconditionally energy decreasing.
\begin{equation}\label{EDETDRK3a}
	\begin{array}{c|ccc}
		0   &         &    &   \\
		\frac{1}{2} &   \frac{1}{2}\phi_{1,2}  &  & \\
		1   &    -\phi_{1,3}  &  2\phi_{1,3}   &       \\
		\hline
		& 4\phi_3-3\phi_2+\phi_1  & -8\phi_3+4\phi_2 & 4\phi_3-\phi_2
	\end{array}
\end{equation}
On the other hand,   all possible types of three-stage third-order ETDRK schemes are listed in  \cite{ETDRK2005a,EXPintegrator}. In fact, many of them with suitable coefficients could satisfy the requirement of our theorem, and  two of them are described with the Butcher Tableaux below:
\begin{equation}\label{EDETDRK3}
    \begin{array}{c|ccc}
    0   &         &    &   \\
    1 &   \phi_{1}  &  & \\
    \frac{2}{3}   &    \frac{2}{3}  \phi_{1,3}-\frac{4}{9}  \phi_{2,3}  &  \frac{4}{9} \phi_{2,3}   &       \\
    \hline
        & \frac{3}{4} \phi_1-\phi_2  & \phi_2-\frac{1}{2} \phi_1 & \frac{3}{4} \phi_1
\end{array}
\end{equation}
and
\begin{equation}\label{EDETDRK3b}
    \begin{array}{c|ccc}
    0   &         &    &   \\
    \frac{4}{9} &   \frac{4}{9} \phi_{1,2}  &  & \\
    \frac{2}{3}   &    \frac{2}{3}  \phi_{1,3}-\phi_{2,3}  &  \phi_{2,3}   &       \\
    \hline
        &  \phi_1-\frac{3}{2} \phi_2  & 0  & \frac{3}{2} \phi_2
\end{array}
\end{equation}
To begin with, we first present numerical evidences. We plot in Fig. \ref{eigenvalueETDRK3} the smallest eigenvalues of $\frac 12(\Delta+\Delta^T)$ for  the three  schemes above, and we observe that the smallest  eigenvalues of (\ref{EDETDRK3}) and (\ref{EDETDRK3b}) are positive, indicating that the conditions of our theorem are satisfied. In contrast, the smallest eigenvalues for the scheme (\ref{EDETDRK3a}) become negative when $Z$ approaches zero, indicating that the conditions of our theorem are violated. To rigorously show that the above two ETDRK schemes decrease the energy, we only need to prove the positive-definiteness of $\Delta'$ for the above schemes, which could be verified by checking all leading principal minors. 
\begin{proposition}\label{prop1}
    For the schemes (\ref{EDETDRK3}) and (\ref{EDETDRK3b}), the determinants of all leading principal minors of $\Delta'$ are positive, and thus $\Delta'$ is positive definite. Thus, the schemes (\ref{EDETDRK3}) and (\ref{EDETDRK3b}) unconditionally decrease the energy for general gradient flows \eqref{sgf}. 
\end{proposition}

\begin{proof}
According to (\ref{EDETDRK3}), we derive the expressions for the determinants of leading principal minors as follows
\begin{equation*}
    \begin{aligned}
        \text{Det}(\Delta'_{1\times 1})= &\Delta'_{11}=\frac{e^z-1}{z}>0
        ;\\
        \text{Det}(\Delta'_{2\times2}) = &\frac{1}{9 z^4}(-9 e^{4z/3} + 18 e^{2z/3} +6z - 6z e^{2z/3} - 18z e^{5z/3} +18z e^{7z/3} -9z^2 e^{2z} -9z^2 e^{4z/3}\\& +6z^2 e^{5z/3} +8z^2 -9);\\
        \text{Det}(\Delta')=&\frac{1}{36 z^6} (252 z + 162 e^{2z} -162 e^{2z/3} + 324 e^{5z/3} - 162e^{8z/3} - 324 e^z + 99z e^{2z} + 54z e^{3z} \\
        &- 81z e^{4z} -180 z e^{2z/3} - 45z e^{4z/3} + 378z e^{5z/3} +54z e^{7z/3} - 468 z e^{8z/3} -9z e^{10z/3} \\
        &+ 270 z e^{11z/3} -72z^2 e^z + 78 z^2 e^{2z} - 18z^2 e^{3z} + 83z^3 e^{2z} +6z^2 e^{2z/3} -54z^2 e^{4z/3} \\
        &+ 288 z^2 e^{5z/3} + 27z^3 e^{4z/3} + 18z^3 e^{5z/3} -54 z^2 e^{7z/3} -240 z^2 e^{8z/3} - 324 z e^z + 66 z^2\\
        &-20 z^3 +162 ).
    \end{aligned}
\end{equation*}
It can be verified by Fig. \ref{ETDRK3detPDproof} that for $z$ with small absolute values,
\begin{equation*}
    \begin{aligned}
        &z^4 \text{Det}(\Delta'_{2\times2}) \geq 0.2, \, z\in [-3,-1], \quad &\text{Det}(\Delta'_{2\times2}) \geq 0.2,\, \forall z\in[-1,0),\\
        &z^6 \text{Det}(\Delta') \geq 0.1, \, z\in [-6,-1], \quad &\text{Det}(\Delta')\geq 0.1, \forall z\in[-1,0),
    \end{aligned}
\end{equation*}
where both $z^4 \text{Det}(\Delta')_{2\times2}$ and $z^6 \text{Det}(\Delta')$ are smooth decreasing functions and both $\text{Det}(\Delta')_{2\times2}$ and $\text{Det}(\Delta')$ are smooth increasing functions. \\
For $z$ with large scales, we use the non-exponential term to control the exponentially small term. Therefore, we denote the non-exponential term in $z^4 \text{Det}(\Delta'_{2\times 2})$ as $$P_2=(6z+8z^2-9)/9$$ and the exponential term as $$R_2=(-9 e^{4z/3} + 18 e^{2z/3} - 6z e^{2z/3} - 18z e^{5z/3} +18z e^{7z/3} -9z^2 e^{2z} -9z^2 e^{4z/3} +6z^2 e^{5z/3})/9.$$ 
When $z<-3$, 
$$P_2(z)>P_2(-3)=6,$$ 
$$|R_2| < (9e^{-4}+18e^{-2}+6\times3 e^{-2}+18\times3 e^{-5}+18\times3 e^{-7}+9\times3^2 e^{-2}+9\times3^2 e^{-4}+6\times3^2 e^{-5})/9<2.03,$$ where we use $\beta\leq |\beta|$ for every single term in $R_2$. Thus, $z^4 \text{Det}(\Delta'_{2\times 2})\geq P_2-|R_2| > 6 -2.03>0$ for all $z<-3$. \\
Similarly we could define the non-exponential term $P_3$ and the exponential term $R_3$ for $z^6 \text{Det}(\Delta)$, and for $z<-6$, we have $P_3(z)>P_3(-6)=148.5$, $|R_3(z)|<86$ and thus $z^6 \text{Det}(\Delta')>P_3(z)-|R_3(z)|>0$. Combining all the above results, we have $\text{Det}(\Delta'_{2\times 2})>0, \text{and } \text{Det}(\Delta')>0$ for all negative $z$.

Therefore, all the determinants of leading principal minors are positive, i.e. $\Delta'$ is positive-definite.

\begin{figure}[H]
	\centering
	\includegraphics[width=14cm]{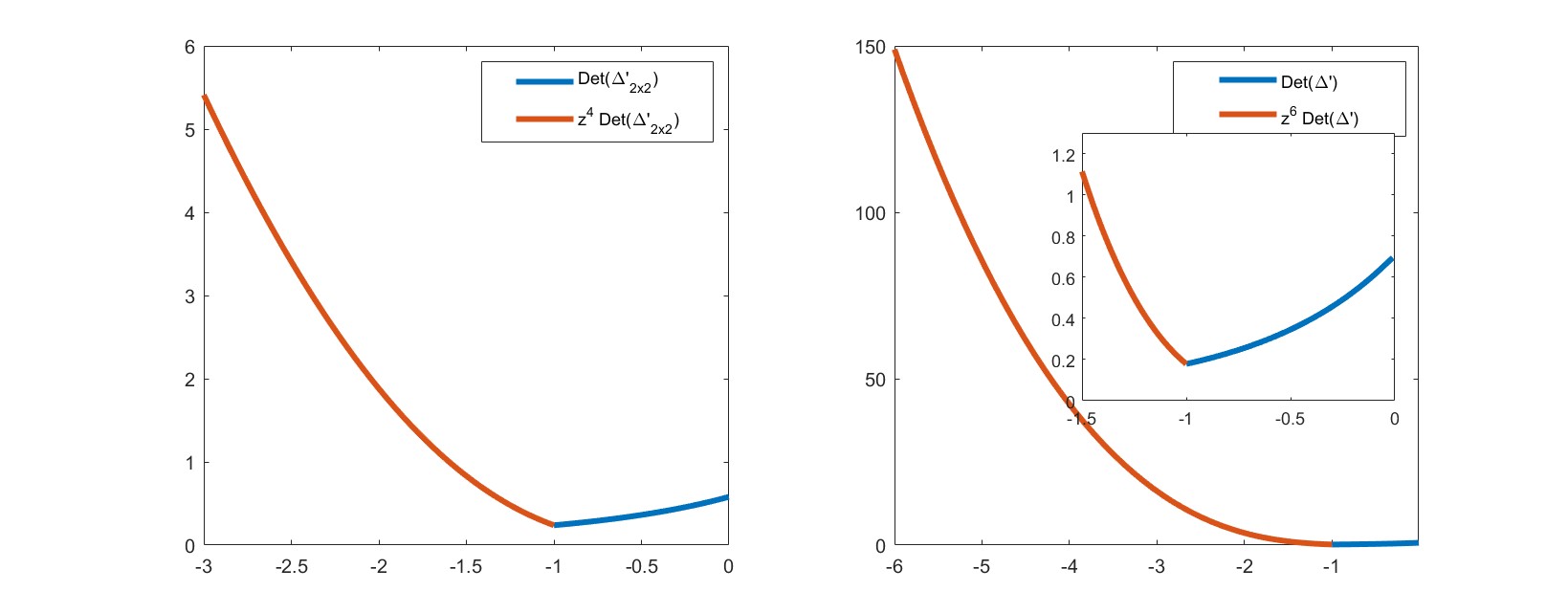}
	\caption{Positive-definiteness test: Determinants of leading principal minors}
	\label{ETDRK3detPDproof}
\end{figure}

For the scheme (\ref{EDETDRK3b}), the proof is the same and we just present the determinants of leading principal minors here.
\begin{equation*}
    \begin{aligned}
        \text{Det}(\Delta'_{1\times 1})= &\Delta'_{11}=\frac{e^{8z/9}-1}{z}>0;\quad \quad
        \text{Det}(\Delta'_{2\times2}) = -\frac{1}{16 z^4}(36z - 162 e^{2z/3} + 81 e^{4z/3} - 36 z e^{2z/3}\\
        &+ 72 z e^{10 z/9} - 72 z e^{16 z/9} + 16 z^2 e^{4z/3} + 16 z^2 e^{8 z/9} + 16 z^2 e^{10z/9} - 12 z^2 + 81);\\
        \text{Det}(\Delta')=&-\frac{1}{32 z^5} (22 e^{2z} - 71 z + 288 e^{2z/3} - 130 e^{4z/3} - 216 e^{5z/3} + 80 e^{7z/3} - 72 e^{8z/3} + 50 e^{10 z/3}\\
        &+ 136 e^z + 5 z e^{2z} + 63 z e^{2z/3} + 16 z e^{4z/3} - 86 z e^{5z/3} + 4 z e^{7z/3} + 23 z e^{8z/3} - 20 z e^{10z/3}\\
        &- 128 z e^{10z/9} - 28 z e^{13z/9} + 144 z e^{16z/9} - 4 z e^{19z/9} + 172 z e^{22z/9} - 156z e^{28z/9} - 4 z^2 e^z\\
        &- 50 z^2 e^{2z} - 30 z^2 e^{4z/3} + 8 z^2 e^{5z/3} - 4 z^2 e^{7z/3} + 56 z^2 e^{8z/3} + 2 z^2 e^{10z/3} - 30 z^2 e^{8z/9}\\
        &- 28 z^2 e^{10z/9} - 8 z^2 e^{13z/9} + 12 z^2 e^{17z/9} + 8 z^2 e^{19z/9} + 40 z^2 e^{22z/9} + 18 z^2 e^{26z/9} \\
        &- 12 z^2 e^{28z/9} + 66 z e^z + 22 z^2 - 158).
    \end{aligned}
\end{equation*}

\end{proof}

 Since the scheme (\ref{EDETDRK3}) has the largest positive smallest eigenvalue according to the following Fig. \ref{eigenvalueETDRK3}, we shall use it to represent energy stable ETDRK3 for numerical tests in Section \ref{sec4}.

\begin{figure}[H]
	\centering
	\includegraphics[width=10cm]{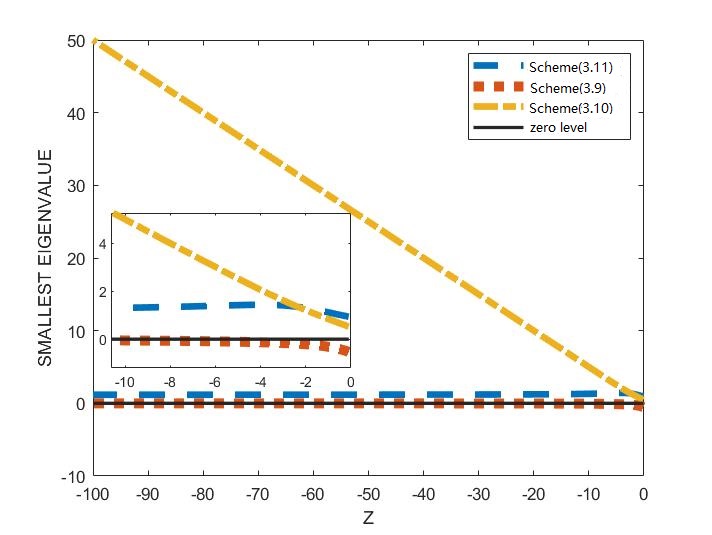}
	\caption{Eigenvalues of $\frac 12(\Delta+\Delta^T)$ for the three ETDRK3 schemes}
	\label{eigenvalueETDRK3}
\end{figure}

\subsection{Fourth-order ETDRK schemes}

In general, fourth-order ETDRK schemes need to satisfy four more order conditions:
\begin{center}\label{ordcond4}
\begin{tabular}{|c|c|}
\hline
  Order & Conditions \\ 
  \hline
  4 & $\psi_4=0$ \\
  4 & $\sum_{i=1}^s b_i J \psi_{3,i}=0$ \\
  4 & $\sum_{i=1}^s b_i J \sum_{j=2}^{i-1} a_{ij} J \psi_{2,j}=0$ \\
  4 & $\sum_{i=1}^s b_i c_i K \psi_{2,i}=0$ \\
  \hline
\end{tabular}
\end{center}
where $\psi_i(z)=\phi_i(z)-\sum_{k=1}^s b_k c_k^{j-1}/(j-1)!, \psi_{i,j}=\phi_i c_j^i-\sum_{k=1}^{j-1} a_{jk} c_k^{i-1}/(i-1)!$, and $J, K$ denote arbitrary bounded operators (for more details, see \cite{EXPintegrator}). These fourth-order conditions are more complicated than previous third-order conditions so that it is more difficult to construct fourth-order ETDRK schemes.  One of the most used ETDRK4 schemes is from Cox and Matthews \cite{ETDstiff}:
\begin{equation}
    \begin{array}{c|cccc}
    0   &    &    &    &   \\
    \frac{1}{2} &    \frac{1}{2} \phi_{1,2} & & &\\
    \frac{1}{2} &  0 & \frac{1}{2} \phi_{1,3}  &  & \\
    1   &    \frac{1}{2}\phi_{1,3} (\phi_{0,3}-1)  & 0 & \phi_{1,3}  &  \\
    \hline
        & \phi_1-3\phi_2+4\phi_3 & 2\phi_2-4\phi_3 & 2\phi_2-4\phi_3 & 4\phi_3-\phi_2
\end{array}
\end{equation}
The other is given by Krogstad \cite{Gintfac}:
\begin{equation}
    \begin{array}{c|cccc}
    0   &    &    &    &   \\
    \frac{1}{2} &    \frac{1}{2} \phi_{1,2} & & &\\
    \frac{1}{2} &  \frac{1}{2} \phi_{1,3} -\phi_{2,3} & \phi_{2,3}  &  & \\
    1   &   \phi_{1,4} -2\phi_{2,4}  & 0 & 2\phi_{2,4}  &  \\
    \hline
        & \phi_1-3\phi_2+4\phi_3 & 2\phi_2-4\phi_3 & 2\phi_2-4\phi_3 & 4\phi_3-\phi_2
\end{array}
\end{equation}
We plot in Fig. \ref{fig:eig2} the smallest eigenvalues of $\frac 12(\Delta+\Delta^T)$ for the two schemes above, the results clearly indicate that  the above two schemes do not satisfy our conditions for energy stability.
\begin{figure}[htbp]
    \centering
    \includegraphics[width=8cm]{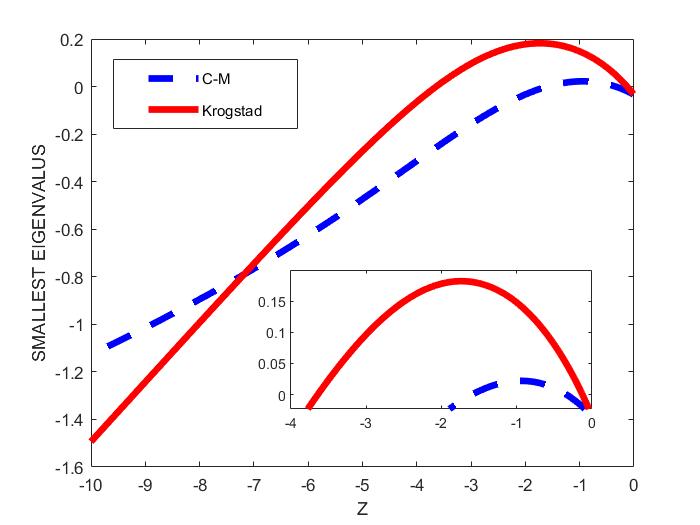}
    \caption{Eigenvalues of ETDRK4 schemes from Cox and Mathhews \cite{ETDstiff}, and Krogstad   \cite{Gintfac}}
    \label{fig:eig2}
\end{figure}
Moreover, we checked all existing fourth-order ETDRK schemes, and also searched from a wide family of four-stage and five-stage fourth-order ETDRK methods, but none of them satisfy the requirement of our theorem. 
Therefore, the existence of energy stable ETDRK4 schemes is still an open problem.

\section{Numerical Experiments}\label{sec4}

In this section we carry out some numerical experiments to illustrate the convergence and energy decay property of our new ETDRK schemes for different phase-field models. We first verify the temporal convergence rates using smooth initial data for the Allen--Cahn and Cahn--Hilliard equations. Next, we study the behavior with large time steps for the Allen--Cahn, Cahn--Hilliard, and phase field crystal equations. Then, we present energy curves for different time steps and stabilization constants to observe their influence. Finally, we present an adaptive time-stepping strategy that takes advantage of the unconditional energy stability.

In all simulations, we consider models with periodic conditions, and use a Fourier-spectral method in space with a sufficiently fine mesh so that the spatial discretization errors can be ignored compared with temporal discretization errors. The ETDRK3 scheme (\ref{EDETDRK3}) is used if not specified otherwise. Besides, we also consider a truncated double-well potential $\Tilde{F}(u)$ so that it naturally satisfies the Lipschitz condition. More precisely, for a sufficiently large $M$ ($M=2$ is enough for cases used in this paper), we replace $F(u)=\frac{1}{4} (u^2-1)^2$ by 
\begin{eqnarray}\nn
    \Tilde{F}(u) &= \left\{\begin{aligned}
        &\frac{3M^2-1}{2} u^2 - 2 \text{sgn}(u)  M^3 u +\frac{1}{4} (3M^4+1), \quad &|u|>M  \\
        &\frac{1}{4} (u^2-1)^2, \quad &|u|\leq M  
    \end{aligned}
    \right.,
\end{eqnarray}
and $f(u)=u^3-u$ by
\begin{eqnarray}\nn
    \Tilde{f}(u)=\Tilde{F}'(u) &= \left\{\begin{aligned}
        &(3M^2-1) u - 2 \text{sgn}(u)  M^3, \quad &|u|>M  \\
        &u^3-u, \quad &|u|\leq M  
    \end{aligned}
    \right..
\end{eqnarray} In fact, the maximum norm of numerical solutions never exceeds the bound $M$ so this replacement does not affect the properties of numerical solutions.
\subsection{Convergence tests}
We solve the Allen--Cahn and Cahn--Hilliard equations in $\Omega=(0,2\pi)\times(0,2\pi)$ with the smooth initial data $u_0=0.5 \sin{x}\sin{y}$. To compute the errors and the convergence rate, we take the number of grid points $N=128$, the interfacial parameter $\epsilon=0.5$ and set the final time $T=0.32$. With these settings we compute the numerical solutions with various time steps $\tau=0.01/2^{k}$ with $k=0,1,...,4$ and calculate the relative errors to get the convergence rate. The results for the Allen-Cahn and Cahn-Hilliard equations are listed in Tables \ref{table1} and \ref{table2}, respectively. In both cases, desired convergence rates are observed. 


\begin{table}[H]
  \begin{center}
    \begin{tabular}{|c|c|c|c|c|} 
    \hline
      \textbf{$\tau$=0.01} & \textbf{$L^\infty$ err} &\textbf{rate} & \textbf{$L^2$ err} & \textbf{rate}\\
      \hline
      $\tau$   & 2.6852e-08 & -      & 2.0736e-09 & -  \\ 
      $\tau$/2 & 3.4044e-09 & 2.9795 & 2.6291e-10 & 2.9795  \\
      $\tau$/4 & 4.2863e-10 & 2.9896 & 3.3101e-11 & 2.9896 \\
      $\tau$/8 & 5.3815e-11 & 2.9936 & 4.1557e-12 & 2.9937 \\
      $\tau$/16 & 6.7815e-12 & 2.9883 & 5.2341e-13 & 2.9891 \\
      \hline
    \end{tabular}
    \caption{ETDRK3 errors and convergence rates for the Allen-Cahn equation}\vspace{-2em}\label{table1}
  \end{center}
  
\end{table}


\begin{table}[H]
  \begin{center}
    \begin{tabular}{|c|c|c|c|c|} 
    \hline
      \textbf{$\tau$=0.01} & \textbf{$L^\infty$ err} &\textbf{rate} & \textbf{$L^2$ err} & \textbf{rate}\\
      \hline
      $\tau$   & 4.2646e-07 & -      & 3.7881e-08 & -  \\ 
      $\tau$/2 & 5.3484e-08 & 2.9952 & 4.8765e-09 & 2.9576  \\
      $\tau$/4 & 6.6767e-09 & 3.0019 & 6.1872e-10 & 2.9785 \\
      $\tau$/8 & 8.3316e-10 & 3.0025 & 7.7913e-11 & 2.9894 \\
      $\tau$/16 & 1.0392e-10 & 3.0032 & 9.7637e-12 & 2.9964 \\
      \hline
    \end{tabular}
    \caption{ETDRK3 errors and convergence rates for the Cahn-Hilliard equation}\vspace{-2em}\label{table2}
  \end{center}
\end{table}

\subsection{Large time step tests}

In this subsection, we study the accuracy of solutions with large time steps. We set $N=128$, $T=8$, $\epsilon=0.1$ for the Allen--Cahn equation and $\epsilon=0.5$ for the Cahn--Hilliard equation. We plot in Fig. \ref{errdtACCH} the relation between the error and the time steps $\tau=2^{1-k}$ for $k=0,1,...,7$, and observe that the solutions with large time steps still maintain good accuracy.

Next, we simulate the phase field crystal model with $N=256$, $\epsilon=0.025$, and $\beta=3$ in $\Omega=(0,32)\times(0,32)$. Using both ETDRK2 and ETDRK3 with the initial condition $u_0=\sin(\frac{\pi x}{16}) \sin(\frac{\pi x}{16})$, we compute the solution at $T=1$ with $\tau=2^{-k}$ for $k=0,1,...,11$. The results are plotted in Fig. \ref{PFCacc}. We observe that  both ETDRK schemes work well with large time steps. More precisely, with $\tau=0.1$, the errors of  ETDRK3 and ETDRK2 are about  $O(10^{-4})$ and $O(10^{-3})$, respectively. 

\begin{figure}[H]
\includegraphics[width=8cm,height=6cm]{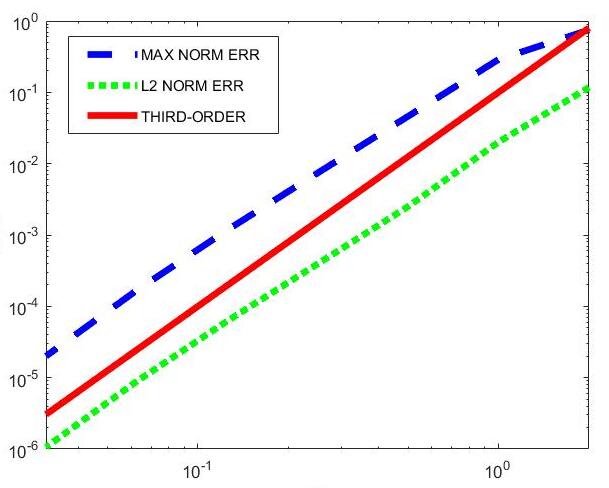}
\includegraphics[width=8cm,height=6cm]{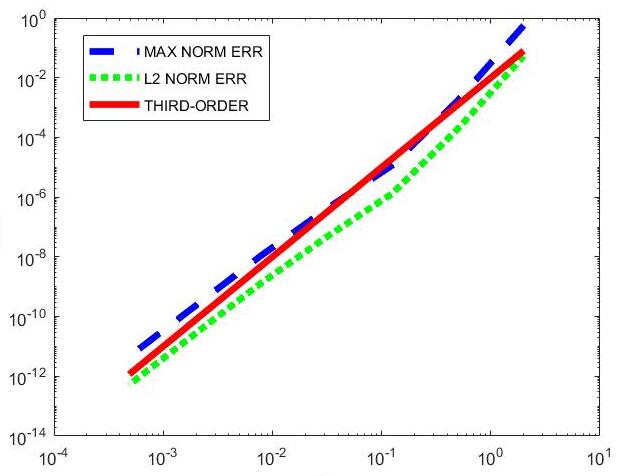}
\caption{Error-$\tau$ figure for the AC (left) and CH (right)}
\label{errdtACCH}
\end{figure}

\begin{figure}[H]
\centering
\includegraphics[width=13cm,height=10cm]{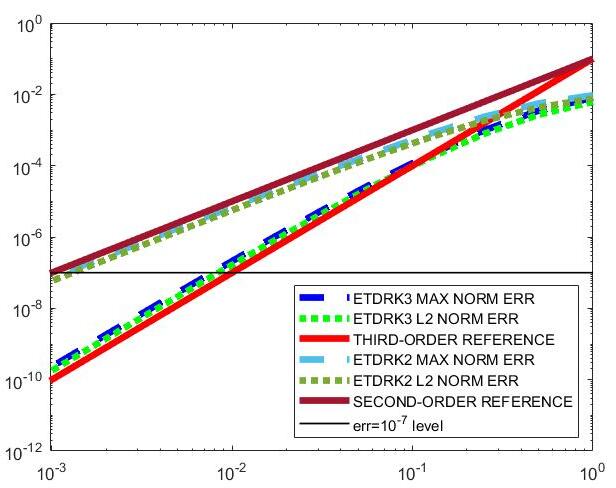}
\caption{Error-$\tau$ figure for the phase-field crystal model}
\label{PFCacc}
\end{figure}

\subsection{Energy curves with different time steps and stabilization constants}

Now we study how energy curves behave with different time steps and stabilization constants. It is well-known that stabilization may affect the dynamics and cause serious time delay phenomena when lower-order schemes are used or large time steps are taken. We first test the Cahn--Hilliard equation with $N=128$, $\epsilon=0.1$ and $\beta=2$. The results are presented in Fig. \ref{EnergyCH}. 
We observe that when time steps are large, the time delay phenomena is obvious, but with  smaller time steps $\tau\leq 0.01$, the influence becomes negligible.

\begin{figure}[H]
    \centering
    \includegraphics[width=12cm, height=8cm]{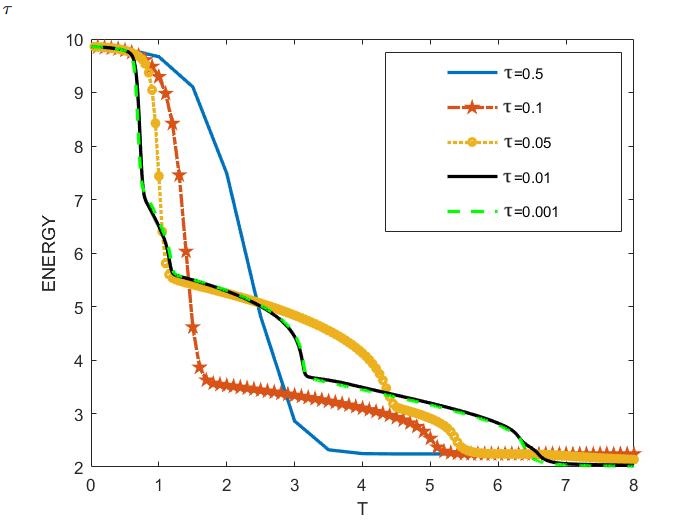}\vspace{-1em}
    \caption{Energy curves with different time steps for the Cahn--Hilliard equation}
    \label{EnergyCH}
\end{figure}

Next we examine the effect of different stabilization constants. Although theoretically they have to be larger than the Lipschitz constant, we can still run the simulation with smaller stabilization constants. While  smaller stabilization constants cause less time delay, but  solutions may lose its accuracy and even blow-up. Fig. \ref{EnergyDTbeta} shows that when $\tau=0.1$, the solutions are all inaccurate and exhibit obvious time delays, with the solution for $\beta=0.1$ even showing nonphysical oscillations. This evidence indicates that the stabilization does help the numerical solutions to maintain stability. When $\tau=0.01$, all three curves with $\beta=0.1,1,2$ are very close, which indicates that the stabilization has very little effect   when the time step is small.
\begin{figure}[htbp]
    \centering
    \includegraphics[width=12cm, height=8cm]{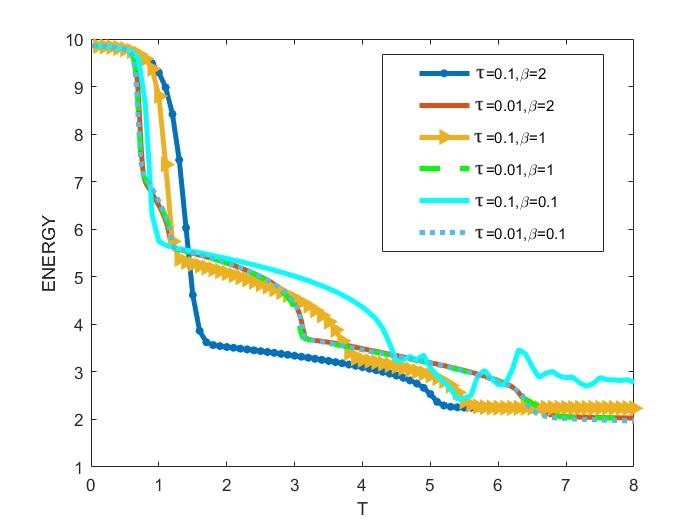}\vspace{-1em}
    \caption{Energy curves with different time steps and stabilization constants for the PFC model}
    \label{EnergyDTbeta}
\end{figure}

\subsection{Long time behavior of the phase field crystal model}
We simulate in this subsection the long time evolution of the phase field crystal (PFC) model  with $N=256$, $\epsilon=0.025$, $\Omega=(0,128)\times(0,128)$, and the initial condition
\begin{equation}
    u_0=0.05+0.01*rand(x), \quad x\in\Omega
\end{equation}
where $rand(x)$ is a uniformly distributed random function satisfying $-1\leq rand(x)\leq 1$. The long time behavior of the solution obtained by ETDRK3 with $\beta=3$ and $\tau=0.1$ is presented in Fig.\ref{PFCdyna}. 
\begin{figure}[H]
    \centering
    \includegraphics[width=14cm]{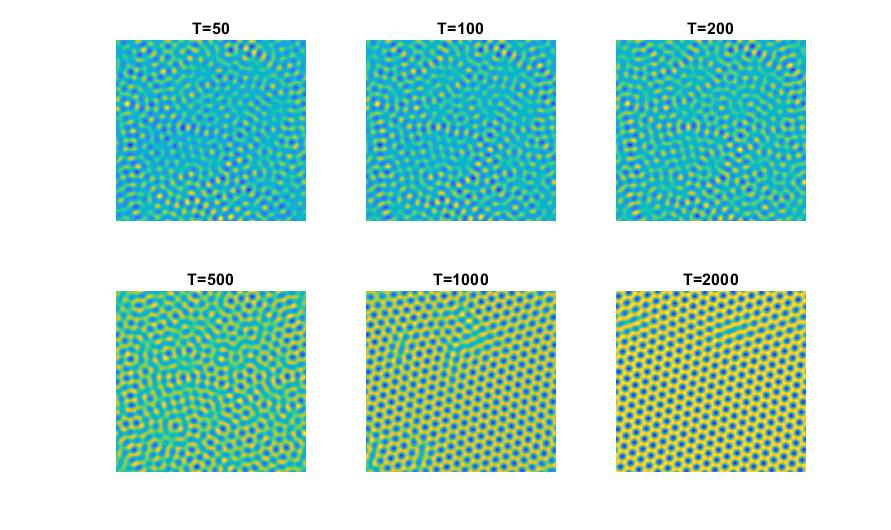}\vspace{-2em}
    \caption{ETDRK3 solutions for the PFC model with $\beta=3, \tau=0.1$ at T=50, 100, 200, 500, 1000, 2000}
    \label{PFCdyna}
\end{figure}
The energy curves with different time steps and stabilization constants are shown in Fig. \ref{PFCenergy}. There are 5 curves in total and 4 of them overlap. With $\beta=3$, the solution is good when $\tau=0.1$, but when $\tau=1$, the time delay caused by the stabilization is much more severe, which indicate that we should not use any stabilization when computing with large time steps. However, it is remarkable that ETDRK3 without stabilization performs very well even with the time step as large as $\tau=10$. 

\begin{figure}[H]
    \centering
    \includegraphics[width=13cm,height=9cm]{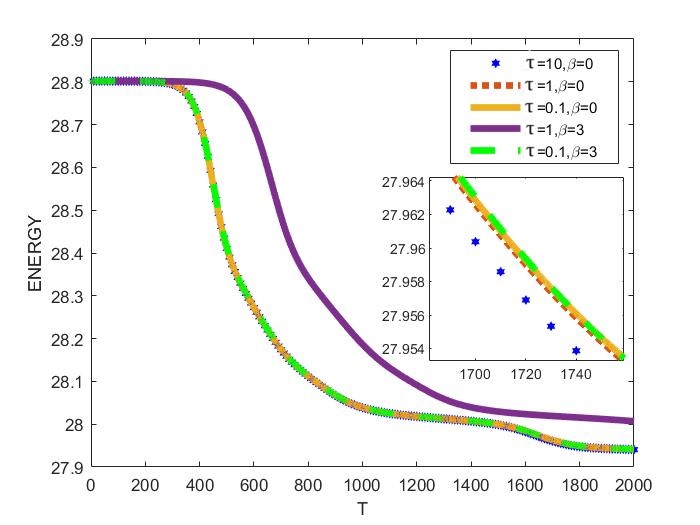}\vspace{-0.5em}
    \caption{Energy curves with different time steps and stabilizers}
    \label{PFCenergy}
\end{figure}

\subsection{Adaptive time stepping}
Solutions of gradient flows may vary drastically during some short time intervals, but change only slightly at other times. A main advantage of unconditional energy stable schemes is that they can be easily used with an adaptive time stepping algorithm, in which the time step is only dictated by accuracy rather than by stability. There are often essential difficulties in applying adaptive time-stepping strategies to other schemes, since most of them do not have robust unconditional stability with variable step sizes. This is also where the significance of the high-order unconditional energy stable schemes lies.

For gradient flows, there are several effective adaptive time stepping strategies, see \cite{BDF2CHvarstep,largestepCH,parafreeCH,adapCH}. Here we make use of the strategy in \cite{adpt} summarized in the following Algorithm 1. In Step 4 and 6, the time step size is given by the formula 
\begin{equation}
    A_{dp} (e,\tau)=\rho \left(\frac{tol}{e}\right)^{r} \tau,
\end{equation}
along with the restriction of the minimum and maximum time steps. In the above formula, $\rho$ is a default safety coefficient, $tol$ is a reference tolerance, $e$ is the relative error computed at each time level in Step 3, and $r$ is the adaptive rate. In our numerical examples, we set $\rho=0.9$ and $tol=5*10^{-3}$, the minimum time step is $10^{-4}$, while the maximum time step is $\tau=10^{-2}$ for Fig. $6$ and Fig. $7$, and $\tau=10^{-1}$ for Fig. $8$ and $9$. The initial time step is taken as the minimum time step.

\begin{algorithm}[H]
\caption{Adaptive time stepping procedure}
\textbf{Given:} $U^n,\tau_n$\\
\textbf{Step 1.} Compute $U^{n+1}_1$ by the first-order ETD scheme with $\tau_n$.\\
\textbf{Step 2.} Compute $U^{n+1}_2$ by the third-order ETDRK scheme with $\tau_n$.\\
\textbf{Step 3.} Calculate $e_{n+1}=\frac{\|U^{n+1}_1-U^{n+1}_2\|}{\|U^{n+1}_2\|}$.\\
\textbf{Step 4.} If $e_{n+1}>tol$, recalculate the time step $\tau_n\xleftarrow[]{} \max\{\tau_{min},\min\{A_{dp}(e_{n+1},\tau_n),\tau_{max}\}\}$,\\
\textbf{Step 5.} goto Step 1.\\
\textbf{Step 6.} else, update the time step $\tau_{n+1}\xleftarrow[]{}\max\{\tau_{min},\min\{A_{dp}(e_{n+1},\tau_n),\tau_{max}\}\}$.\\
\textbf{Step 7.} endif
\end{algorithm}

We take the two-dimensional Cahn--Hilliard equation as an example to examine the performance of the adaptive time stepping algorithm. We take $\epsilon=0.1$, $N=512$, $\beta=2$ and $r=1/3$. As comparison, we compute two ETDRK3 solutions with a small uniform time step $\tau=10^{-4}$ and large uniform time step $\tau=10^{-2}$ as references.

 We plot in Fig. \ref{CHadapenergy1}  the energy curses (left) and  the size of adaptive time steps (right). We observe that solution obtained with $\tau=10^{-2}$ is not accurate, while the adaptive time stepping solutions are in excellent agreement with the small time step solution with $\tau=10^{-4}$.  In addition, except at the initial time period  and two other small time intervals where the energy varies drastically, the adaptive time steps basically stay around $\tau=10^{-2}$. Finally, we plot in  Fig.  \ref{CHadapdyna1} snapshots of the phase evolution at different times.

\begin{figure}[H]
\centering    
\includegraphics[width=14cm]{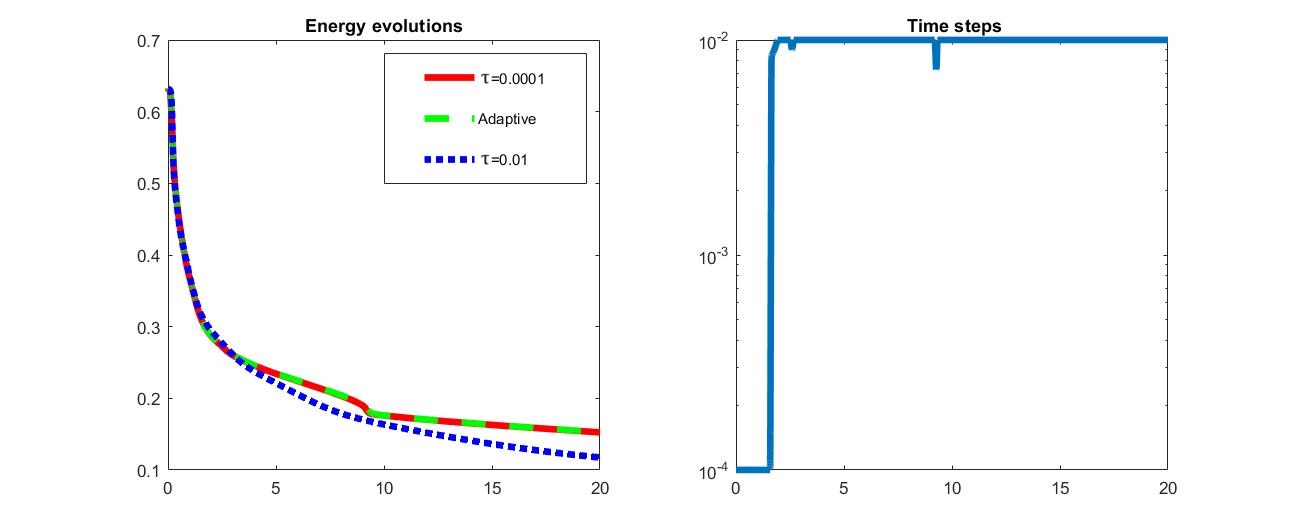}\vspace{-0.5em}
\caption{Energy curves among small time steps $\tau=0.0001$, adaptive time steps and large time steps $\tau=0.01$ and the size of time steps in the adaptive procedure}
\label{CHadapenergy1}
\end{figure}

\begin{figure}[H]
\centering         
\includegraphics[width=14cm]{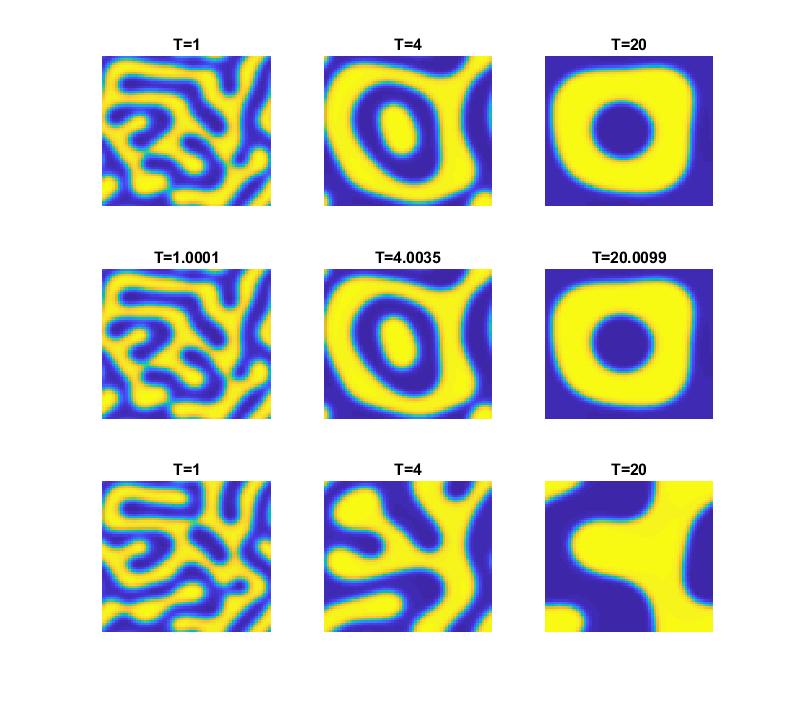}\vspace{-2em}
\caption{Solutions for the Cahn--Hilliard equation using small time steps $\tau=0.0001$ (first line), adaptive time steps (second line) and large time steps $\tau=0.01$ (third line)}
\label{CHadapdyna1}
\end{figure}

%

\section{Concluding Remarks}\label{sec5}

In this paper, we presented a general framework for constructing unconditionally energy stable ETDRK schemes of arbitrary order for a class of gradient flows and identified a set of conditions that must be satisfied for an ETDRK scheme to be unconditionally energy stable. In particular, we showed that the widely employed third- and fourth-order ETDRK schemes are not unconditionally energy stable, and constructed new  third-order ETDRK schemes which are unconditionally energy stable.

To the best of our knowledge, this is the first rigorous result on unconditionally energy stable ETDRK schemes higher than second-order. Potential future extensions include:
\begin{itemize}
    \item Higher than third-order schemes: While we provided a set of conditions that ETDRK schemes need to satisfy in order to be unconditionally energy stable, but due to the complexity of these conditions,  the existence of higher than third-order ETDRK  unconditionally energy stable schemes is still an open question. One possibility is to add   one more stage in the RK part which will lead to   more free parameters to choose from. 
    \item Other models: We restricted ourselves to a class of gradient flows with Lipschitz nonlinearities. Since the ETDRK schemes are based on the Duhamel's Principle so the only error comes from the numerical integration, we expect that the new ETDRK schemes would behave well for other models with even more complicated physical structures, although proving rigorously unconditionally energy stability for more complicated models would be a challenge.
\end{itemize}
We are hopeful that our results in this paper can be extended to more general systems  and to ETDRK schemes higher than third-order.

\bigskip
\noindent {\bf Acknowledgments}

J. Shen is partially supported by NSFC 12371409.
J. Yang is supported by the National Natural Science Foundation of China (NSFC) Grant No. 12271240, the NSFC/Hong Kong RGC Joint Research Scheme (NSFC/RGC 11961160718), and the Shenzhen Natural Science Fund (No. RCJC20210609103819018).

%

%

\bibliography{ref}
\bibliographystyle{siamplain}

\end{document}